\documentclass[12pt,reqno]{amsart}
\usepackage[utf8]{inputenc}
\usepackage[english]{babel}
\usepackage{comment}
\usepackage{amsmath,amssymb,amsfonts,amsbsy,amsthm,amscd,latexsym,graphicx}
\usepackage{empheq,textcomp}
\usepackage{xcolor}
\usepackage{hyperref}

\theoremstyle{definition}
\newtheorem{theorem}{Theorem} [section]
\newtheorem{corollary}[theorem]{Corollary}
\newtheorem{lemma}[theorem]{Lemma}

\newtheorem{definition}[theorem]{Definition}

\newtheorem{remark}[theorem]{Remark}
\newtheorem{example}[theorem]{Example}

\numberwithin{equation}{section}

\AtEndDocument{
  \par
  \bigskip
  \begin{tabular}{@{}l@{}}%
    \text{Department of Mathematics, University of Oregon, Eugene, OR 97403-1222, USA}\\
    \textit{E-mail address}: \texttt{putingyu@uoregon.edu}\\
  \end{tabular}}

\newcommand{\C}{\mathbb{C}}

\newcommand{\Ec}{{\mathcal{E}}}
\newcommand{\Fc}{{\mathcal{F}}}
\newcommand{\Gc}{{\mathcal{G}}}
\newcommand{\N}{\mathbb{N}}
\newcommand{\R}{\mathbb{R}}

\newcommand{\Z}{\mathbb{Z}}

\newcommand{\Eq}{\, = \,}

\newcommand{\Le}{\, \le \,}

\newcommand{\qeddef}{{\quad $\diamondsuit$}}
\newcommand{\qeddeff}{{\qquad \diamondsuit}}

\newcommand{\bigabs}[1]{\bigl|\,#1\,\bigr|}

\newcommand{\ip}[2]{\langle\,#1,#2\,\rangle}
\newcommand{\bigip}[2]{\bigl\langle \,#1, \, #2 \,\bigr\rangle}

\newcommand{\norm}[1]{\|\,#1\,\|}
\newcommand{\bignorm}[1]{\bigl\|\,#1\,\bigr\|}
\newcommand{\Bignorm}[1]{\Bigl\|\,#1\,\Bigr\|}

\newcommand{\bigparen}[1]{\bigl(\,#1\,\bigr)}
\newcommand{\Bigparen}[1]{\Bigl(\,#1\,\Bigr)}

\newcommand{\set}[1]{\{#1\}}
\newcommand{\bigset}[1]{\bigl\{#1\bigr\}}

\newcommand{\clspan}{{\overline{\text{span}}}}

\newcommand{\inN}{_{n\in\N}}
\newcommand{\sumli}{\sum_{n=1}^\infty}

\begin{document}
\title{Every semi-normalized unconditional Schauder frame in Hilbert spaces contains a frame}
\author{Pu-Ting Yu}



\date{\today}

\pagestyle{plain}
\maketitle
\begin{abstract}
Let $H$ be an infinite-dimensional Hilbert space.
 We prove that every unconditional Schauder frame for $H$ contains a subsequence that can be normalized to form a frame for $H$. As a consequence, every semi-normalized unconditional Schauder frame contains a frame for $H.$ Here we say that a sequence $\set{x_n}\inN$ in a Hilbert space $H$ is an \emph{unconditional Schauder frame} for $H$ if there exists some sequence $\set{y_n}\inN\subseteq H$ such that 
    $$x=\sumli \ip{x}{y_n}x_n\quad \text{for all }x\in H,$$
    with the unconditional convergence of the series in the norm of $H.$ We say that $\set{x_n}\inN$ is semi-normalized if $m\leq \norm{x_n}\leq M$ for all $n\in \N$ for some positive constants $m,M.$

We then apply our main results to answer several open questions concerning the existence of certain unconditional Schauderf frames. For example, we prove that if a closed subspace of $L^2(\R^d)$ contains $\set{e^{2\pi ib\cdot x}g}_{b\in \Lambda}$ for some infinite uniformly discrete subset $\Lambda$ of $\R^d$ and some nonzero function $g$ in the Feichtinger algebra, it does not admit any unconditional Schauder frames of translates with finitely many generators. We will also show that no Gabor system with the critical lower Beurling density can be an unconditional Schauder frame when the window function belongs to the Feichtinger algebra. Furthermore, we present an example of a compact set of $\R$ which does not admit any unconditional Schauder frames of exponentials with the critical lower Beurling density. All results in this paper apply equivalently to sequences that can be rescaled to form a frame for $H.$ 
\end{abstract}

\section{Introduction}
Let $X$ be a separable infinite-dimensional Banach space. A sequence $\set{x_n}\inN$ is said to admit a \emph{reconstruction formula} if for every $x\in X$ there exists a sequence of scalars $(c_n)\inN$ such that \begin{equation}
\label{reconstruction_formula}
    x=\sumli c_nx_n\quad \text{for all }x\in X
\end{equation}
with the convergence of the norm in $X.$ The concept  of reconstructing elements using such an infinite linear combinations associated with a sequence can be traced back to \cite{JS27}, where Schauder introduced the notion of Schauder bases an infinite-dimensional analogue of Hamel bases. Here we say $\set{x_n}\inN$ is a \emph{Schauder basis} if the series representation (\ref{reconstruction_formula}) is unique for all $x\in X.$
The question as to which sequences admit a reconstruction formula, and under what conditions such a formula possesses specific convergence properties, have been pivotal in various branches of mathematics, including Banach space theory, approximation theory, harmonic analysis, et al., as well as in the field of applied mathematics. To be more precise, the problem of constructing sequences that admit a reconstruction formula with desired properties can be further refined into several layers of inquiry: 
(I) Control of coefficients: Beyond the validity of the reconstruction formula, does there exist an associated sequence $\set{y_n}\subseteq X^*$ such that those coefficients $c_n$ can be realized as $y_n(x)$ for all $n\in\N$?; (II) Unconditionality of the summation: The convergence of the series in Equation (\ref{reconstruction_formula}) may fail under a different enumeration of the sequence $\set{x_n}\inN$. This naturally leads to ask under what conditions is the convergence of the series in Equation (\ref{reconstruction_formula}) independent of the order of the summation?; (III) Uniqueness of the reconstruction formula: Under what conditions is the series representation (\ref{reconstruction_formula}) unique for every $x\in X?$ Freeman et al.\ proved in \cite{FPT21} that one can always construct a sequence that admits a reconstruction formula satisfying (I) out of a dense subset of $X$ if $X$ has a Schauder basis.
However, it has been shown, in general, constructing sequences that has a specific form and admits a reconstruction formula satisfying (I) and (II) is a highly nontrivial task (\cite{OSSZ11},\cite{FOSZ14},\cite{LT25b},\cite{LT25c}). In fact, it has been shown in various settings that the construction of sequences of certain forms that admit a reconstruction formula fulfilling (I)--(II) or (I)--(III) is impossible (\cite{OSSZ11},\cite{FOSZ14},\cite{LT25},\cite{LT25d},\cite{DH00},\cite{PY25a}).

In the setting of Hilbert spaces, Duffin and Schaeffer introduced a novel approach in \cite{DS52} to constructing sequences (what they called as \emph{frames}) that admits a reconstruction formula satisfying (I) and (II). A sequence $\set{x_n}\inN$ in a Hilbert space $H$ is a frame if there exist positive constants $A$ and $B$, called frame bounds, such that \begin{equation}
\label{frame_ineq}
A\, \norm{x}^2
\Le \sum_{n=1}^\infty |\ip{x}{x_n}|^2
\Le B\, \norm{x}^2,
\qquad\text{for all } x \in H.
\end{equation}
A sequence $\set{x_n}$ is said to be a \emph{Bessel sequence} in $H$ if at least the upper inequality of Equation (\ref{frame_ineq}) is satisfied. It was shown in the same paper that for every frame $\set{x_n}\inN$ there exists another sequence $\set{y_n}\inN$, called the \emph{canonical dual frame}, such that 
\begin{equation}
\label{frame_reconstruction_formula}
x=\sumli\ip{x}{y_n}x_n\quad \text{for all }x\in H,
\end{equation}
with the unconditional convergence of the series in the norm of $H$. Here we say an infinite series converges unconditionally if it converges regardless of the order of the summation. Since its rediscovery in the 1980s, a substantial body of research has investigated the existence of frames with specific forms (\cite{KN15},\cite{ACMT17},\cite{BZ20},\cite{BKL23}, see also \cite{Hei11},\cite{Chr16}\,and\,\cite{OU16} for relatively recent textbook recountings and references therein). Nevertheless, it has also been shown that frames of certain forms do not exist (\cite{CDH99},\cite{AFK14},\cite{KNO23},\cite{PY25c},\cite{PY25d},\cite{EV25}). As a consequence, the search for sequences of the same form with ``slightly weaker" frame-like properties has emerged as another major question in this direction.
Sequences in Hilbert spaces that admit a reconstruction formula can be further classified into four families, as illustrated by the hierarchical structure below
\begin{equation}
\label{hier_structure}
\text{Riesz bases} ~\subseteq~ \text{Frames}  ~\subseteq~ 
\begin{tabular}{l}
    \text{ Unconditional}  \\
     \text{Schauder frames}
\end{tabular}~\subseteq~ \text{Schauder frames}.
\end{equation}
Here \emph{Riesz bases} are frames that are also Schauder bases. We call a sequence $\set{x_n}\inN$ a \emph{Schauder frame} for $H$ if there exists some sequence $\set{y_n}\inN$, called \emph{associated coefficient functionals}, such that the series in Equation (\ref{frame_reconstruction_formula}) converges in the norm of $H.$ 
If the convergence of the series in Equation (\ref{frame_reconstruction_formula}) is unconditional, then we say $\set{x_n}\inN$ is an \emph{unconditional Schauder frame.}
The inclusion relationships between these families then naturally raise the following question: if a specific form of frame does not exist, does there exist an unconditional Schauder frame for $H$ of that form? Our main result below shows that, in general, the answer is negative.
\begin{theorem}
\label{non_exist_frame_implies_unschfr}
    Let $S\subseteq H$ be a subset for which there exist some positive constants $m,M$ such that $m\leq \norm{x}\leq M$  for all $x\in S$. Assume that there exists no frame for $H$ consisting elements from $S$. Then there does not exist any unconditional Schauder frame for $H$ consisting elements from $S$. \qeddef
\end{theorem}
Theorem \ref{non_exist_frame_implies_unschfr} is optimal in the sense that the assumption that $m\leq \norm{x}\leq M$  for all $x\in H$ is indispensable. For example, let $S=\set{\frac{1}{n}e_n,ne_n}\inN$, where $\set{e_n}\inN$ is an orthonormal basis for $H.$ Then there exist no frame consisting of elements of $S$, both $\set{\frac{1}{n}e_n}\inN$ and $\set{ne_n}\inN$ form unconditional Schauder frames for $H.$ 

Due to absence of the frame inequality (\ref{frame_ineq}), there is typically a technical gap when attempting to disprove the existence of an unconditional Schauder frame of certain form. With our main result, we are now able to resolve several open problems regarding the existence of certain unconditional Schauder frame for $H$. It was conjectured by Olson and Zalik in \cite{OZ92} that no sequence in $L^2(\R^d)$ consisting of pure translates of a nonzero function can be a Schauder basis for $L^2(\R^d).$ This conjecture remains open at the time of writing. It was recently confirmed by Lev and Tselishchev in \cite{LT25} that there does not exist any unconditional Schauder frame for $L^2(\R^d)$ consisting of translates of finitely many functions. We improve their result to a broader family of closed subspaces of $L^2(\R^d).$ 
For further background on the Olson–Zalik conjecture and the terminology used in the theorem below, see Section \ref{OZconjecture}. We say that a subset $\Gamma$ of $\R^d$ is \emph{uniformly discrete} if $\inf\limits_{\substack{a,b\in\Gamma}}\norm{a-b}\geq \delta $ for some $\delta>0$ for any two distinct $a,b$.
\begin{theorem}
    Let $g$ be a nonzero function in the Feichtinger algebra $M^1(\R^d)$. Suppose that $M\subseteq L^2(\R^d)$ is a closed subspace containing $\set{e^{2\pi ib\cdot x}g(x)}_{b\in \Gamma}$ for some infinite uniformly discrete subset $\Gamma\subseteq\R^d$. Then $M$ does not admit any unconditional Schauder frame consisting of translates of finitely many functions. \qeddef
\end{theorem}

The lower and upper \emph{Beurling density} associated with a countable subset $\Lambda\subseteq\R^d$ are defined by 
\begin{equation}
\label{Beruling_density}
D^{-}(\Lambda)=\liminf_{r\rightarrow\infty}\inf_{x\in \R^d}\frac{\#\bigparen{\Lambda\cap B_r(x)}}{|B_r(x)|}\quad\text{and}\quad D^{+}(\Lambda)=\limsup_{r\rightarrow\infty}\sup_{x\in \R^d}\frac{\#\bigparen{\Lambda\cap B_r(x)}}{|B_r(x)|},
\end{equation}
respectively. 
It was raised by Olevskii in \cite{Ol21} whether, for every $S\subseteq\R^d$ of finite measure, there exists some countable subset $\Lambda\subseteq \R$ with $D^{-}(\Lambda)=D^{+}(\Lambda)=|S|$ such that $\set{e^{2\pi i\lambda x}}_{\lambda\in \Lambda}$ forms a frame for $L^2(S).$ Enstad and van Velthoven proved that the
asnwer is negative in \cite{EV25} recently. We further show that the answer remains negative for unconditional Schauder frames.

\begin{theorem}
    There exists some compact set $S$ of $\R$ of positive measure such that no countable subset $\Lambda \subseteq\R$ with $D^{-}(\Lambda)=|S|$ yields an unconditional Schauder frame of the form $\set{e^{2\pi i\lambda x}}_{\lambda\in \Lambda}$ for $L^2(S).$ \qeddef
\end{theorem}

We also prove that no Gabor system at critical lower density (see Section \ref{applications} for the definition) can simultaneously be an unconditional Schauder frame and possess generators in the Feichtinger algebra.

\begin{theorem}
   Let $\phi_1,\dots,\phi_N$ be functions in the Feichtinger algebra.  There does not exist any countable subsets $\Lambda_1,\dots,\Lambda_N\subseteq \R^{2d}$ with $D^{-}\bigparen{\cup_{i=1}^N\Lambda_i}=1$ such that $$\set{e^{2\pi ib\cdot x}\phi_i(x-a)\,|\,(a,b)\in \Lambda_i, 1\leq i\leq N}$$
   forms an unconditional Schauder frame for $L^2(\R^d)$. Here $\cup_{i=1}^N\Lambda_i$ denotes the disjoint union of $\Lambda_1,\dots\Lambda_N.$  \qeddef
\end{theorem}

It has been shown that the unconditional convergence property endows unconditional Schauder frames with partial frame properties (\cite{Or33},\cite{CCSL02},\cite{SB13},\cite{BFPS24}). As a result, it is generally believed (and conjectured) that one can construct a frame from an unconditional Schauder frame by suitable means. For example, it was conjecture by Stoeva and Balazs in \cite{SB13} that every unconditional Schauder frame and its associated functionals can be simultaneously rescaled to form a frame for $H.$ This conjecture was confirmed to be true by Tselishchev in \cite{AT25} recently. Our theorem below first completely demystifies the constitution of semi-normalized unconditional Schauder frames in Hilbert spaces by showing that they are constructed from a frame together with a collection of additional elements. Here we say a sequence $\set{x_n}\inN$ in $H$ is semi-normalized if there exist some positive constants $m,M$ such that $m\leq \norm{x_n}\leq M$ for all $n\in\N.$ Second, we provides an explicit description of the sequence of scalars needed to rescale a general unconditional Schauder frame into a frame. 
\begin{theorem}
    Every unconditional Schauder frame $\set{x_n}\inN$ for $H$ contains a subsequence that can normalized to form a frame for $H.$ That is, there exists some subsequence $\set{x_{n_k}}_{k\in\N}$ of $\set{x_n}\inN$ such that $\bigset{\frac{x_{n_k}}{\norm{x_{n_k}}}}_{k\in\N}$ is a frame for $H.$ 

    In particular, if $\set{x_n}\inN$ is semi-normalized, then it must contain a frame for $H.$ \qeddef
\end{theorem}
This paper is organized as follows.  
In Section \ref{preliminaries}, we introduce necessary notations and definitions and present several known results required for this paper. 
Our main results are then established in Section \ref{main_results}. Finally, we conclude this paper by presenting more applications of our main results in Section \ref{applications}.

\section{Preliminaries}
\label{preliminaries}
Throughout this paper, we denote by $H$ an infinite-dimensional separable Hilbert space equipped with the inner product $\ip{\cdot}{\cdot}.$ For $x\in H$, we write $\norm{x}$ instead of $\norm{x}_H$ for simplicity. We will write $\norm{v}$ to mean the Euclidean norm of $v$ when $v\in \R^d.$ Whenever necessary, the norm of an element in a specific Hilbert space will be indicated explicitly. For any countable set $J$, we use $\#J$ to denote the number of elements in $J.$ The (Lebesgue) measure of a subset $S$ of $\R^d$ is denoted by $|S|.$

We will use the selector form of Weaver's $\text{KS}_2$ conjecture established by Bownik in \cite{Bo24}. Follow the definition and terminology in \cite{Bo24}, we define the notion of \emph{binary selectors} as follows. For any $N\in\N$, we denote by $\set{0,1}^N$ the set of all $N$-tuples where each components are either $0$ or $1$. 
\begin{definition}
     Let $I$ be a countable index set and let $\set{J_k}_{k\in J}$ be any partition of $I$ with $\#J_k=2$ for all $k\in J$. Binary selectors of order $1$ are sets $I_0$ and $I_1$ such that $I = I_0 \cup I_1$ and $$\#(J_k\cap I_0)=\#(J_k\cap I_1)=1  \quad \text{for all } k\in J.$$
For $N\geq 2$, we define binary selectors of order $N$ by induction. Assume that $\set{I_b}_{b\in \set{0,1}^{N-1}}$ are binary selectors of order $N-1$. For each $b\in \set{0,1}^{N-1}$, let $\set{J_{b,k}}_{k\in J}$ be any partition of $I_b$ with $\#J_{b,k}=2$ for all $k\in J$. Then binary selectors of order $N$ are sets $I_{b,0}$ and  $I_{b,1}$ satisfying $I_b=I_{b,0}\cup I_{b,1}$ and $\#(I_{b,0}\cap J_{b,k})=\#(I_{b,1}\cap J_{b,k})=1$ for all $k\in J.$ \qeddef 
\end{definition}
We say that a linear operator $T\colon H\rightarrow H$ is \emph{positive (semi-definite)} if $\ip{Tx}{x}\geq 0$ for all $x\in H.$
A positive trace-class operator $T$ on $H$ is a positive operator on $H$ for which $$\text{tr}(T)\Eq\sumli \ip{Te_n}{e_n}<\infty$$ for some orthonormal basis $\set{e_n}\inN$ for $H.$ We now state the selector form of Weaver's $\text{KS}_2$ conjecture as follows.
\begin{theorem}(\cite[Theorem 5.3]{Bo24})
\label{binary_selector_theorem}
    Let $\delta>0$ and let $\set{T_n}\inN$ be a family of positive trace-class operators defined on $H$. Assume that $$T\coloneq \sumli T_n\Le \textbf{I} \quad \text{ and }\quad\text{tr}(T_n)\leq \delta\text{ for all }n\in\N.   $$
    Then there exists some absolute constant $C>0$ such that for any $N\in\N$ with $2^N<\frac{1}{\delta}$ and any intermediate choices of partitions with sets of size $2$ there exist binary selectors $I_b$, $b\in \set{0,1}^N$, that form a partition of $\N$ satisfying \begin{align}
    \Bignorm{2^N\sum_{n\in I_b}T_n-T}\Le C\sqrt{2^N\delta},  \end{align}
 for all $b\in\set{0,1}^N$. \qeddef
\end{theorem}
We remark that Theorem \ref{binary_selector_theorem} also holds for finite collection of positive trace-class operators by considering zero operators. 
Moreover, fix $\delta>0$ and define the sequence of scalars $(B_n)_{n=0}^\infty$ by 
$$B_0=1, \quad B_{j+1}=B_j+4\sqrt{2^j\delta B_j}+2^{j+1}\delta,~j\geq1.$$
The absolute constant $C$ in Theorem \ref{binary_selector_theorem} may be chosen so that $\sum_{j=0}^{N-1}(B_j-1)\Le C\sqrt{2^N\delta},$
for every $N\in \N$ such that $2^N\delta<1$ (See \cite[Lemma 5.2]{Bo24} or \cite[Lemma 10.20]{OU16} for a proof). 

We will also need the following elementary result. For a proof, see \cite[Lemma 5.2]{Bo24}. 
\begin{lemma}
\label{elementary_Hilbert_result}
Let $T$ be a positive operator defined on $H$. For any given closed subspace $M$ of $H$, let $P_M$ be the orthogonal projection onto $M$. Then for any closed subspace $M$ of $H$ there exists some constant $C_M= \bigparen{\norm{P_MTP_M}~\norm{P_{M^\perp}TP_{M^\perp}}}^{1/2}$ such that  
$$-C_M\textbf{I}\Le T-P_MTP_M-P_{M^\perp}TP_{M^\perp}\Le C_M\textbf{I} \qeddeff$$
\end{lemma}

Another key ingredient in our proof is the following theorem due to Tselishchev.
\begin{theorem}(\cite[Theorem 1.2]{AT25})
\label{unsch_scalable_lemma}
Let $\set{x_n}\inN$ be an unconditional Schauder frame for $H$ with coefficient functionals $\set{y_n}\inN$.
Then there exists some sequence of nonzero scalars such that $\set{c_n}\inN$ such that both $\set{c_nx_n}\inN$ and $\set{\bar{c}_n^{-1}y_n}$ are frames for $H.$\qeddef


\end{theorem}

\section{Main Results}
\label{main_results}
In this section, we prove our main result. To this end, we require the following two lemmas. The first one is a refinement of \cite[Lemma 7.3]{Bo24}, where we provide an universal bound on the number of duplicates that the sampling function could take.     
\begin{lemma}
\label{sampling_lemma}
Let $I$ be a finite set.
Let $\set{T_n}_{n\in I}$ be a family of positive trace-class operator on $H$ with $\text{tr}(T_n)\leq \delta$ for some $\delta>0$ for all $n\in I$ and let $(c_n)_{n\in I}$ be a set of positive numbers such that 
$$T\coloneq\sum_{n\in I}c_nT_n\leq \frac{1}{2}\textbf{I}.$$
Suppose that $M\subseteq H$ is a closed subspace such that $\gamma\coloneq \text{tr}(P_{M}TP_M)\leq 1$. Then for any $0<\epsilon<1$ there exists a finite set $I'$ and a sampling function $\sigma\colon I'\rightarrow I$ such that $$\frac{\epsilon}{2} P_{M^\perp}-6\sqrt{\gamma}\textbf{I}\leq \frac{1}{2^\beta}\sum_{n\in I'}T_{\sigma(n)}-T\leq \frac{\epsilon}{2} P_{M^\perp}+6\sqrt{\gamma}\textbf{I}, $$
where $\beta\in \N$ is such that 
$1<2^{\beta}\frac{\epsilon^2}{4C^2\delta}\leq 2$. Moreover, we have $$\#\set{k\in I'\,|\,\sigma(k)=n}\leq 2^{\beta+1}c_n \quad\text{for all $n\in I$}.$$ 
\end{lemma}
\begin{proof}
    For each $n\in I$ we let $(\ell_{n,j})_{j\in\N}\subseteq \N$ be a sequence such that $c_n=\sum_{j=1}^\infty 2^{-\ell_{n,j}}.$ Then for each $0<\epsilon<\min(1,C\sqrt{8\delta})$ we choose $L_n$ so large that $$T-\sum_{n\in I}\sum_{j=1}^{L_n} 2^{-\ell_{n,j}}T_{n}\leq \min(\epsilon/{2},\gamma)\textbf{I}$$
Let  $\set{m_{n,j}\,|\,n\in I,1\leq j\leq M_n}$ be a set of scalars that $$\sum_{j=1}^{L_n}2^{-\ell_{n,j}}+\sum_{j=1}^{M_n}2^{-m_{n,j}}=\Big\lceil \sum_{j=1}^{L_n}2^{-\ell_{n,j}}\Big\rceil \quad \text{for all }n\in I,$$
where $\lceil\cdot\rceil$ is the ceiling function. Let $\beta\in \N$ be such that $1<2^{\beta}\frac{\epsilon^2}{4C^2\delta}\leq 2$ where $C>0$ is the absolute constant described in Theorem \ref{binary_selector_theorem}. We then define $\eta=\max\limits_{n,j}(\ell_{n,j},m_{n,j},\beta)$ and a family of positive trace-class operators $\set{\phi_n}_{n\in I}$ 
that satisfies\begin{enumerate} \setlength\itemsep{0.5em}
    \item [\textup{(i)}] $\clspan\set{\phi_n}\subseteq\clspan\set{T_n}$,
\item [\textup{(ii)}] $\sum_{n\in I_0}\sum_{j=1}^{M_n}2^{-m_{n,j}}\phi_n\leq \frac{1}{2}\textbf{I},$
\item [\textup{(iii)}] $\text{tr}(\phi_n)\Le 2^{-\beta+2}\max\set{\epsilon,\gamma},$
\end{enumerate}
for all $n\in I$.
For each $n\in I, 1\leq j\leq L_n$ we decompose $2^{-\ell_{n,j}}T_n$ and $2^{-m_{n,j}}\phi_n$ into $$2^{-\ell_{n,j}}T_n= \underbrace{2^{-\eta}T_n +\cdots+2^{-\eta}T_n}_{2^{\eta-\ell_{n,j}} \text{ terms}}=\sum_{i=1}^{2^{\eta-\ell_{n,j}}}T_{n,i},$$
and 
$$2^{-m_{n,j}}\phi_n= \underbrace{2^{-\eta}\phi_n +\cdots+2^{-\eta}\phi_n}_{2^{\eta-m_{n,j}} \text{ terms}}=\sum_{i=1}^{2^{\eta-m_{n,j}}}\phi_{n,i},$$
where $T_{n,i}$ and $\phi_{n,i}$ denote $2^{-\eta}T_n$ and $2^{-\eta}T_n$, respectively. Note that $\text{tr}(T_{i,n})\leq2^{-\eta}\delta$ for all $n,i.$ Let $$I_1=\set{(n,i)\,|\,1\leq i\leq 2^{\eta-\ell_{n,j}},1\leq j\leq L_n,n\in I}$$ and 
$$I_2=\set{(n,i)\,|\,1\leq i\leq 2^{\eta-m_{n,j}},1\leq j\leq M_n,n\in I}$$ 
Let $N=\eta-\beta$. We construct binary selectors of order $N$ as follows. Note that $2^{-N}\bigparen{\#I_1+\#I_2}\in \N$. Consequently, it is guaranteed that we execute such a binary selecting process at least $N$ time. At first, we partition $I_1$ into sets of size $2$ by pairing $(n,i)$ that belong to $I_1$ and share the same first index $n$. For those $(n,i)\in I_1$ that is left unmatched, we randomly pair it with $(n,j)\in I_2.$ Next, we randomly pair the rest of $(n,j)\in I_2.$ Let $B_b$ be an arbitrary binary selector of order $1$. We partition $B_b$ with the same logic and then extract the binary selectors of order $2$. We then repeat this procedure $N$ times and obtain binary selectors of order $N.$ Note that for each $k\in I$ we have $$\#\set{(n,i)\in I_1\,|\,n=k\,}\Le 2^{\eta}\sum_{j=1}^{L_k}2^{-\ell_j}\leq 2^\eta c_k.$$
It follow that \begin{equation}
\label{cardinality_control}
    \#\bigparen{B_b\cap \set{(n,i)\in I_1\,|\,n=k\,}}\Le 2^{1+\eta-N} c_k=2^{\beta+1} c_k
\end{equation} for all binary selectors of order $N$.

By Theorem \ref{binary_selector_theorem}, there exist some subsets $J_1\subseteq I_1$ and $J_2\subseteq I_2$ such that 
\begin{equation}
\label{binary_selector_estimate}
\Bignorm{2^N\Bigparen{\sum_{(n,i)\in J_1}T_{n,i}+\sum_{(n,i)\in J_2}\phi_{n,i}}-\sum_{n\in I}\sum_{j=1}^{L_n} (2^{-\ell_{n,j}}T_{n}+2^{-m_{n,j}}\phi_{n})}\Le  C\sqrt{2^{-\beta}\delta}<\frac{\epsilon}{2}.
\end{equation}
Moreover, note that the collection $\{B_b\,|\, b\in \{0,1\}^N\}$ forms a partition of $J_1\cup J_2$. Consequently, we have
\begin{align}
\begin{split}
\sum_{b\in \{0,1\}^N} \sum_{(n,i)\in B_{b}}\bigparen{T_{n,i}+\phi_{n,i}}
&\Eq 2^{-\eta}\sum_{b\in \{0,1\}^N} \sum_{(n,i)\in B_{b}}\bigparen{T_{n}+\phi_{n}}\\ &\Eq \sum_{n\in I}\Bigparen{\sum_{j=1}^{L_n}2^{-\ell_{n,j}}T_n+\sum_{j=1}^{M_n}2^{-m_{n,j}}\phi_n}
\end{split}
\end{align}
Using the additivity of trace and $\#\bigparen{\set{0,1}^N}=2^N$, we see that there exists at least one binary selector $B_b$ such that $$\text{tr}\Bigparen{2^N\sum_{(n,i)\in B_{b}}\bigparen{T_{n,i}+\phi_{n,i}}}\Le \text{tr}\Bigparen{\sum_{n\in I}\Bigparen{\sum_{j=1}^{L_n}2^{-\ell_{n,j}}T_n+\sum_{j=1}^{M_n}2^{-m_{n,j}}\phi_n}}=2\gamma.$$ 
Thus, we can even pick $J_1,J_2$ in Equation (\ref{binary_selector_estimate}) to be the subsets such that
 \begin{align}
\begin{split}
  \label{trace_dominate_operator_norm_ineq}  
\text{tr}\Bigparen{P_{M}\Bigparen{2^N\bigparen{\sum_{(n,i)\in J_1}T_{n,i}+\sum_{(n,i)\in J_2}\phi_{n,i}}}P_{M}}\leq 2\gamma.
\end{split}
\end{align}
 For notational convenience, we define $$\Phi\coloneq 2^N\bigparen{\sum_{(n,i)\in J_1}T_{n,i}+\sum_{(n,i)\in J_2}\phi_{n,i}} \quad \text{and}\quad \Psi\coloneq \sum_{n\in I}\sum_{j=1}^{L_n} (2^{-\ell_{n,j}}T_{n}+2^{-m_{n,j}}\phi_{n})$$
Using the fact that 
\begin{align}
\begin{split}
\bignorm{P_{M^\perp}\Phi P_{M^\perp}}\Le \bignorm{\Phi}\leq \norm{T}+\epsilon<2,
\end{split}
\end{align}
it follows that, by Lemma \ref{elementary_Hilbert_result}, 
\begin{equation}
\label{estimate_1}
-\sqrt{2\gamma}\textbf{I}\Le \Phi-P_{M}\Phi P_{M}-P_{M^\perp}\Phi P_{M^\perp}\Le \sqrt{2\gamma}\textbf{I}.
\end{equation}
On the other hand, by assumption, we also have 
\begin{equation}
\label{estimate_2}
-\sqrt{\gamma}\textbf{I}\Le \Psi-P_{M}\Psi P_{M}-P_{M^\perp}\Psi P_{M^\perp}\Le \sqrt{\gamma}\textbf{I}.
\end{equation}
Note that by Equation (\ref{binary_selector_estimate}) and Equation (\ref{trace_dominate_operator_norm_ineq}) we have
\begin{equation}
\label{estimate_3}
-\frac{\epsilon}{2} P_{M^\perp} \Le  P_{M^\perp}\bigparen{\Phi-\Theta}P_{M^\perp} \Le  \frac{\epsilon}{2} P_{M^\perp},
\end{equation}
and 
\begin{equation}
\label{estimate_4}
- 2\gamma P_{M} \Le  P_{M}\bigparen{\Phi-\Psi}P_{M} \Le 2\gamma P_{M}.
\end{equation}
Combining Equation (\ref{estimate_1})--(\ref{estimate_4}), it follows that  
$$-(1+\sqrt{2})\sqrt{\gamma}\textbf{I}-\frac{\epsilon}{2}P_{M^\perp}-2\gamma P_M\Le \Phi-\Psi\Le (1+\sqrt{2})\sqrt{\gamma}\textbf{I}+\frac{\epsilon}{2}P_{M^\perp}+2\gamma P_M$$
Since $\gamma<1$, we obtain

$$-(3+\sqrt{2})\sqrt{\gamma}\textbf{I}-\frac{\epsilon}{2}P_{M^\perp}\Le \Phi-\Psi\Le (3+\sqrt{2})\sqrt{\gamma}\textbf{I}+\frac{\epsilon}{2}P_{M^\perp}.$$
By the construction of $\phi_n$ and the fact that $T_{n,i}=2^{-\eta}T_n$, it follows that 
$$-\frac{\epsilon}{2} P_{M^\perp}-6\sqrt{\gamma}\textbf{I}\leq \frac{1}{2^\beta}\sum_{(n,i)\in J_1}T_{n}-T\leq \frac{\epsilon}{2} P_{M^\perp}+6\sqrt{\gamma}\textbf{I}.$$
Let $I'=\set{(n,i)\,|\,(n,i)\in J_1}$ and define $\sigma\colon I'\rightarrow I$ by $\sigma(n,i)=n.$
It is then equivalent that
$$-\frac{\epsilon}{2} P_{M^\perp}-6\sqrt{\gamma}\textbf{I}\leq \frac{1}{2^\beta}\sum_{n\in I'}T_{\sigma(n)}-T\leq \frac{\epsilon}{2} P_{M^\perp}+6\sqrt{\gamma}\textbf{I}.$$
In particular, we have $\#\set{n\in I'\,|\,\sigma(n)=k}\leq 2^{\beta+1}c_k$ for all $k\in I$ by Equation (\ref{cardinality_control}).


\end{proof}
We remark the orthogonal decomposition method utilized in the middle of following proof has been used in \cite{FS19},\cite{Bo24} and \cite{BY26}.
\begin{lemma}\label{upper_bdd_cardinality_control}
    Let $\set{x_n}\inN\subseteq H$ be a sequence and let $(c_n)\inN$ be a sequence of scalars such that $\set{c_nx_n}\inN$ is a frame for $H$ with frame bounds $A$ and $B$. Then there exists a sampling function $\sigma\colon \N\rightarrow\N$ such that $\set{x_{\sigma(n)}}\inN$ can be normalized to form a frame for $H.$
    
    In particular, we have $$\#\set{k\in\N \,|\,\sigma(k)=n}\leq \max(144C^2BA^{-2}, 64C^4B^{-2})|c_n|^2\norm{x_n}^2$$
for all $n\in\N$, where $C$ is the absolute constant described in Theorem \ref{binary_selector_theorem}. 
\end{lemma}
\begin{proof}
By taking absolute value if necessary, we may assume that all $c_{n}$ are positive numbers.
Let $x_n'=\frac{x_n}{\norm{x_n}\sqrt{B}}$
  and let $T_n\colon H\rightarrow H$ be the positive trace-class operator defined by $$T_n(x)\coloneq \ip{x}{x_{n}'}x_{n}' \quad \text{for all }x\in H \text{ and all }n\in\N.$$
  Note that $\text{tr}(T_n)=\frac{1}{B}.$
    We then define $S\colon H\rightarrow H$ to be the frame operator associated with $\set{B^{-1/2}c_{n}x_{n}}_{n\in\N}.$
    That is,   
    $$S\coloneq\sum_{n\in \N} c^2_{n}\norm{x_n}^2T_n\leq \textbf{I}.$$ 
    Let $\epsilon=\min\bigparen{\frac{A}{3B},\frac{\sqrt{B}}{2C}}$, where $C$ is the absolute constant described in Theorem \ref{binary_selector_theorem}. Let $K_0=0$, $K_1=1$ and let $H_1=\set{0}$. We then construct a sequence of subspaces $\set{H_j}_{j\in \N}$ of $H$ inductively as follows. For $j\geq 1$ we define $H_{j+1}$ by 
    $$H_{j+1}=\clspan\set{P_{(H_1\oplus H_2\cdots \oplus H_j)^\perp}T_n(H)\,|\,1\leq n\leq K_j},$$
    and choose $K_{j+1}\in \N$ large enough that 
    \begin{equation}
    \label{choice_of_Kn}
    \text{tr}\Bigparen{P_{(H_1\oplus H_2\cdots \oplus H_{j
+1})}\Bigparen{\sum_{n=K_{j+1}}^\infty 
     c^2_{n}\norm{x_n}^2T_n}P_{(H_1\oplus H_2\cdots \oplus H_{j
    +1})}}\leq \eta_{j+1},
    \end{equation}
    where $\eta_0=0$ and $\eta_j=6^{-2}4^{-j}\epsilon^2$ for $j\geq 1$.
Note that $T_n(H)\subseteq H_1\oplus H_2 \oplus \cdots \oplus H_{j+2}$ for all $1\leq n\leq K_{j+1}$. Let $M_j=(H_{j+1}\oplus H_{j+2})^\perp$ for $j\geq 0.$ Note that $P_{M_0}T_1P_{M_0}$ is the zero operator. 
Now, since $$P_{M_j}\bigparen{\sum_{n=K_{j}+1}^{K_{j+1}} c_{n}^2\norm{x_n}^2T_{n}}P_{M_j}\Eq P_{H_1\oplus\cdots \oplus H_j}\bigparen{\sum_{n=K_{j}+1}^{K_{j+1}} c_{n}^2\norm{x_n}^2T_{n}}P_{H_1\oplus\cdots \oplus H_j},$$
it follows that 
\begin{align}
\label{final_trace_estimate}
\begin{split}
\text{tr}\Bigparen{P_{M_j}\Bigparen{\sum_{n=K_{j}+1}^{K_{j+1}} c_{n}^2\norm{x_n}^2T_{k}}P_{M_j}}&\Eq \text{tr}\Bigparen{P_{H_1\oplus\cdots \oplus H_j}\Bigparen{\sum_{n=K_{j}+1}^{K_{j+1}} c_{n}^2\norm{x_n}^2T_{n}}P_{H_1\oplus\cdots \oplus H_j}}\\
&\Le \text{tr}\Bigparen{P_{H_1\oplus \cdots \oplus H_{j
}}\Bigparen{\sum_{n=K_{j+1}}^\infty 
     c_{n}^2\norm{x_n}^2T_n}P_{H_1\oplus \cdots \oplus H_{j
    }}}\\
    &\Le \eta_{j},
\end{split}
\end{align}
for all $j\geq 0.$ For each $j\geq 0$ let $I_j=\set{n\,|\,K_{j}+1\leq n\leq K_{j+1}}.$
Then by Lemma \ref{sampling_lemma} and Equation (\ref{final_trace_estimate}), there exists some finite set $I_j'$, some $\beta\in \N$, and some sampling function $\sigma_j:I_j'\rightarrow [K_j+1,K_{j+1}]$ such that 
\begin{equation}
\label{operator_relationship_inequality}    
-\frac{\epsilon}{2} P_{M_j^\perp}-6\sqrt{\eta_j}\,\textbf{I}\leq \frac{1}{2^\beta}\sum_{k\in I_j'}T_{\sigma_j(k)}-\sum_{n\in I_j}c^2_{n}\norm{x_n}^2T_{n}\leq \frac{\epsilon}{2} P_{M_j^\perp}+6\sqrt{\eta_j}\,\textbf{I},
\end{equation}
 where $\beta\in\N$ is an universal constant such that $1<2^{\beta}\frac{B\epsilon^2}{4C^2}\leq 2.$ Moreover, we have 
\begin{equation}
\label{cardinality_upper_bound}
\#\Lambda_n\coloneq\#\bigset{k\in I'_j\,|\,\sigma_j(k)=n}\leq 2^{\beta+1}c_n^2\norm{x_n}^2,
\end{equation}
for all $K_{j}+1\leq n\leq K_{j+1}.$ Since $\set{c_nx_n}\inN$ is a frame for $H$, we have $$\sum_{j=0}^\infty P_{M_j^\perp}=\sum_{j=0}^\infty P_{H_{j+1}\oplus H_{j+2}}=2\textbf{I}$$
It then follows from Equation (\ref{operator_relationship_inequality}) that 
$$-2\epsilon\textbf{I}\Le \frac{1}{2^\beta}\sum_{j\geq 0}\sum_{k\in I_j'}T_{\sigma_j(k)}-S\leq 2\epsilon\textbf{I}.$$
By re-enumerating $\cup_{j\geq 0} I_j'$, it follows that there exists some sampling function $\sigma\colon\N\rightarrow \N$ such that 
$\bigset{ \frac{x_{\sigma(n)}}{\norm{x_{\sigma(n)}}}}_{n\in\N}$ 
is a frame for $H$ with frame bounds $2^{\beta}\frac{A}{3}$ and $3\cdot 2^\beta B$. Moreover, since $\epsilon=\min\bigparen{\frac{A}{3B},\frac{\sqrt{B}}{2C}}$, we have \[\#\set{k\in\N \,|\,\sigma(k)=n}\leq \max(144C^2BA^{-2}, 64C^4B^{-2})|c_n|^2\norm{x_n}^2\] for all $n\in\N.$  \qedhere

\qedhere

\end{proof}

We are now ready to prove our main theorem as follows. It is worth noting that unconditional Schauder frames are sometimes defined as a pair of sequences such that Eqaution (\ref{frame_reconstruction_formula}) holds with the unconditional convergence of the series in the norm of $H$ for all $x\in H$ in literature. The advantage of formulating unconditional Schauder frames in terms of a single sequence is that it allows one to establish the following equivalence between unconditional Schauder frames and sequences that can be rescaled to form a frame for $H$ via Theorem \ref{unsch_scalable_lemma}.
\begin{theorem}
\label{characterization_uncondition_Scf}
    Let $\set{x_n}\inN$ be a sequence in $H$. The following statements are equivalent:
\begin{enumerate}
\setlength\itemsep{0.5em}
    \item [\textup{(a)}] $\set{x_n}\inN$ is an unconditional Schauder frame for $H$.
\item [\textup{(b)}] There exists some sequence of scalars $(c_n)\inN$ such that $\set{c_nx_n}\inN$ is a frame for $H$.
    
    \item [\textup{(c)}] There exists some sequence $\set{y_n}\inN\subseteq H$ such that $$x\Eq\sumli\ip{x}{x_n}\,y_n\quad \text{for all }x\in H,$$
    with the unconditional convergence of the series in the norm of $H.$
    \item [\textup{(d)}] There exists a subsequence $(n_k)_{k\in I}\subseteq \N$ such that $\bigset{\frac{x_{n_k}}{\norm{x_{n_k}}}}_{k\in I}$ is a frame for $H$ with $x_{n_k}\notin \text{span}\set{x_{n_j}}$ for any distinct $k,j\in I.$
\end{enumerate}
In particular, if $\set{x_n}\inN$ is semi-normalized, then $\set{x_n}\inN$ contains a frame for $H$ provided one of the conditions (a)--(d) holds.
\end{theorem}
\begin{proof} 
(a)$\Leftrightarrow$(b) By Theorem \ref{unsch_scalable_lemma}, it remains to prove the direction from (b) to (a). Let $\set{y_n}\inN$ be the canonical dual frame associated with $\set{c_nx_n}\inN$. It follows that $\set{x_n}\inN$ is an unconditional Schauder frame for $H$ (with coefficient functionals $\set{\overline{c_n}y_n}\inN).$ \\
\smallskip

(a)$\Leftrightarrow$(c) The equivalence between (a) and (c) was established in \cite[Corollary 3.9]{SB13}.\\
\smallskip

(d)$\Rightarrow$(a) Likewise, let $\set{y_k}_{k\in I}$ be the canonical dual frame associated with $\bigset{\frac{x_{n_k}}{\norm{x_{n_k}}}}_{k\in I}$. Let $\set{z_k}\inN$ be the sequence with $z_k=y_k\norm{x_{n_k}}^{-1}$ if $k\in I$ and $z_k=0$ if $n\notin I.$ It then follows that 
$$x\Eq\sum_{k\in I}\ip{x}{y_{k}}\frac{x_{n_k}}{\norm{x_{n_k}}}\Eq \sum_{k\in \N}\ip{x}{z_k}\,x_k\quad \text{for all }x\in H,$$
with the unconditional convergence of series in the norm of $H.$\\
\smallskip

(a)$\Rightarrow$(d) 
 For each $k\in \N$ we define $\Gamma_k=\bigset{x_n\,|\,x_n\in\text{span}\set{x_k}}.$
    Let $n_1=1$ and let $(n_k)_{k\in\N}\subseteq \N$ be an increasing sequence such that $$\set{x_n}\inN=\bigcup_{k\in \N}\Gamma_{n_k}\quad \text{and}\quad \Gamma_{n_k}\cap \Gamma_{n_j}=\emptyset\quad\text{if } k\neq j.$$
Without loss of generality, we may assume that $x_{n_k}\neq 0$ for all $k\in\N$. By Theorem \ref{unsch_scalable_lemma}, there exists some sequence of scalars such that $\set{c_{n,k}x_{n_k}}_{n\in \Gamma_{n_k},\,k\in\N}.$ Now by Lemma \ref{upper_bdd_cardinality_control}, there exists a sampling function $\sigma\colon \N\rightarrow \N$ such that $\set{x_{n_{\sigma(k)}}}_{k\in\N}$ can be normalized to form a frame for $H.$ In particular, $$\#\Lambda_k\coloneq\#\set{j\in\N \,|\,\sigma(j)=k}\leq \max(144C^2BA^{-2}, 64C^4B^{-2})\sum_{n\in \Gamma_{n_k}}|c_{n,k}|^2\norm{x_{n_k}}^2,$$
where $A,B$ are frame bounds of $\bigset{c_{n,k}x_{n_k}}_{k\in\N}$ and $C$ is the absolute constant described in Theorem \ref{binary_selector_theorem}. Then for each $k\in \N$ we have $$\norm{x_{n_k}}^4\sum_{n\in \Gamma_{n_k}}|c_{n,k}|^2\Eq\sum_{n\in \Gamma_{n_k}}|\ip{x_{n_k}}{c_{n,k}x_{n_k}}|^2\Le B\norm{x_{n_k}}^2.$$
It follows that $$\sum_{n\in \Gamma_{n_k}}|c_{n,k}|^2\norm{x_{n_k}}^2<B\quad \text{for all }k\in\N.$$
Let $L=\max(144C^2BA^{-2}, 64C^4B^{-2}).$ Statement (d) then follows from the inequality
$$L^{-2}\sum_{k\in\N}\bigabs{\bigip{x}{\frac{x_{n_{\sigma(k)}}}{\norm{x_{n_{\sigma(k)}}}}}}^2\Le \sum_{k\in \text{Range}(\sigma)}\bigabs{\bigip{x}{\frac{x_{n_{k}}}{\norm{x_{n_{k}}}}}}^2\Le \sum_{k\in\N}\bigabs{\bigip{x}{\frac{x_{n_{\sigma(k)}}}{\norm{x_{n_{\sigma(k)}}}}}}^2.$$

Finally, assume that there exists $0<m\leq M$ such that $$m\Le \norm{x_n}\Le M\quad \text{for all }n\in\N.$$ Since 
$$ M^{-2} \sum_{k\in \text{Range}(\sigma)}|\bigip{x}{x_{n_k}}|^2\Le \sum_{k\in \text{Range}(\sigma)}\bigabs{\bigip{x}{\frac{x_{n_{k}}}{\norm{x_{n_{k}}}}}}^2\Le m^{-2} \sum_{k\in \text{Range}(\sigma)}|\bigip{x}{x_{n_k}}|^2,$$
it follows that $\set{x_{n_k}}_{k\in \text{Range}(\sigma)}$ is a frame for $H.$
\end{proof}


\begin{remark}
 (a) In general, one should not expect an arbitrary unconditional Schauder frame for $H$ to contain a frame for $H$. For example, both $\set{n^{-1}e_n}\inN$ and $\set{ne_n}\inN$ are unconditional Schauder frames for $H$, but neither of them contains a frame for $H.$ However, by Lemma \ref{upper_bdd_cardinality_control}, if a sequence can be rescaled to form a frame by using sufficiently small sequence of scalars, then it contains a frame for $H$. \medskip

(b) 
Theorem \ref{characterization_uncondition_Scf} (d) provides a clearer characterization of the scalars chosen to rescale  unconditional Schauder frames into a frame than Theorem \ref{unsch_scalable_lemma} when the notion of an unconditional Schauder frame is formulated as a single sequence. Specifically, for any unconditional Schauder frame  $\set{x_n}\inN$ for $H$ there always exists a subset $I\subseteq \N$ such that one can always rescale $\set{x_n}\inN$ into a frame by using the rescaling scalars $c_n=\norm{x_n}^{-1}$ if $n\in I$ and $c_n=0$ if $n\notin I$. Let $\set{y_n}\inN$ be the sequence of coefficient functionals associated with $\set{x_n}\inN$,
we remark that it remains unclear to us whether there exists some subset $J\subseteq\N$ such that both $\bigset{\frac{x_n}{\norm{x_n}}}_{n\in J}$ and $\set{\norm{x_n}y_n}_{n\in J}$ are frames for $H.$ By Orlicz's Theorem (\cite{Or33}), $\set{\norm{x_n}y_n}_{n\in J}$ must be a Bessel sequence.  However, the existence of a lower frame bound for $\bigset{\norm{x_n}y_n}_{n\in J}$ remains unknown. 
\qeddef
\end{remark}

\section{Applications}
\label{applications}
This section is devoted to several applications of our main results. We begin with applications to Gabor systems. Let $g_1,\dots g_N$ be nonzero functions in $L^2(\R^d)$ and let $\Gamma_1,\dots\Gamma_N\subseteq \R^{2d}$ be countable subsets. The associated \emph{Gabor system} (or \emph{Weyl-Heisenberg system}) $\Gc_N(\phi_i,\Gamma_i)$ is the subset $$\Gc_N(\phi_i,\Gamma_i)\Eq \bigset{M_bT_a\phi_i\,|\,(a,b)\in \Gamma_i,1\leq i\leq N},$$
where $M_b\colon L^2(\R^d)\rightarrow L^2(\R^d)$ is the modulation operator defined by $(M_bg)(x)=e^{2\pi ib\cdot x}g(x)$ and $T_a\colon L^2(\R^d)\rightarrow L^2(\R^d)$ is the translation operator defined by $(T_ag)(x)=g(x-a).$ The functions $\phi_1,\dots,\phi_N$ are usually called the \emph{window functions} or \emph{generators} and the sets $\Gamma_1,\dots,\Gamma_N$ are the associated sets of time-frequency shifts.
Unless otherwise specified, we write $\Gc_N(\phi_i,\Gamma_i)$ to denote the Gabor system generated by functions $\phi_1,\dots,\phi_N$ in $L^2(\R^d)$ and some countable subsets  $\Gamma_1,\dots\Gamma_N$ in $\R^{2d}.$ In the case $N=1$, we will simply write $\Gc(\phi,\Gamma).$

Next, we provide a brief introduction to \emph{modulation spaces}. 
\begin{definition}\label{Modulation_definition}
Fix a nonzero Schwartz function $\psi\in S(\R)$. 
\begin{enumerate}\setlength\itemsep{0.5em}
    \item [\textup{(a)}] The \emph{short-time Fourier transform} of $f\in L^2(\R^d)$, denoted by $V_{\psi}f$, is the complex-valued measurable function on $\R^{2d}$ defined by  $$V_{\psi}f(x,w)\Eq \overline{\ip{M_wT_x \psi}{f}} \Eq  \ip{f}{M_wT_x \psi}.$$
     \item [\textup{(b)}] For $1\leq p \leq 2$, the \emph{modulation space} $M^{p}(\R^d)$ is the space consisting of 
      $f\in L^2(\R^d)$ for which $\norm{V_\psi f}_{L^{p}(\R^{2d})}$ is finite, i.e.,
     $$M^{p}(\R^d)=\bigset{f \in L^2(\R^d)\,\big|\,  \norm{f}_{M^{p}(\R^d)}\Eq \norm{V_\psi f}_{L^{p}(\R^{2d})}<\infty}.$$     
    \end{enumerate}
    In particular, the modulation space $M^1(\R^d)$ is usually called \emph{Feichtinger algebra} in literature.
    \qeddef
\end{definition} 

 The discovery of modulation spaces can be traced back to early 1980s by Feichtinger. Through a series of subsequent collaborations between Feichtinger and Gr\"{o}chenig, these spaces have been recognized as fundamental to the development of time–frequency analysis. Loosely speaking, modulation spaces are a family of function spaces which classify functions based on their joint time and frequency localization (via short-time Fourier transform).  We remark that, in light of the main focus of this paper, the definition of modulation spaces that is presented above is merely a specific family of modulation spaces commonly found in the literature.
 For more details and related applications of modulation spaces, we refer to \cite{Gro01} and \cite{BO20}.

\subsection{Olson-Zalik Conjecture} Olson and Zalik formulated the conjecture in \cite{OZ92} that no system of translates generated by a function in $L^2(\R^d)$ can be a Schauder basis. Here a system of translates generated by finitely many functions $g_1,\dots, g_N\subseteq L^2(\R^d)$ is the set $\set{T_{a_i}g_i\,|\,a_i\in\Lambda_i,\,1\leq i\leq N}$ for some $\Lambda_1,\dots,\Lambda_N\subseteq\R^d.$ Consequently, systems of translates are just special cases of Gabor systems in which no frequency shifts are involved. In the same paper where this conjecture was proposed, they showed that no systems of translates generated by a single generator can be a Riesz basis for $L^2(\R)$. Christensen, Heil and Deng then proved in \cite{CDH99} that there does not exist any frame of translates generated by finitely many function in $L^2(\R^d).$ Lev and Tselishchev further proved that no system of translates with finitely many generators can be an unconditional Schauder frame for $L^p(\R^d)$ for any $1\leq p\leq 2$ in \cite{LT25}. 

Regarding subspaces of $L^2(\R^d)$ that do not admit Riesz bases, frames or unconditional Schauder frames that are formed by systems of translates, it was shown by the author in \cite{PY25a} that no modulation space $M^p(\R)$ with $1<p\leq 2$ admits an unconditional Schauder frame of translates with a single generator. Assuming that $S$ is an infinite-dimensional closed subspace of $L^2(\R^d)$ that contains a nonzero function in $M^1(\R^d)$ and is closed under modulation (i.e., there exists some nonzero $b\in \R^d$ such that $M_b f\in S$ for some nonzero $b\in\R$ and all $f\in S$), the author proved in \cite{PY25c} that no such closed subspace admits a frame of translates generated by finitely many functions. We apply our main result below to characterize a broader class of subspaces that fail to admit unconditional Schauder frames of translates generated by finitely many functions.

\label{OZconjecture}
\begin{theorem}
\label{distance_criterion}
    Fix $1\leq p\leq 2.$
Let $S$ be an infinite-dimensional closed subspace of $L^2(\R^d)$. Let $\Gc(\phi,\Lambda)$ be a Bessel sequence in $S$ with $\phi \in M^p(\R^d)$. Assume that $\clspan\bigset{\Gc(\phi,\Lambda)}$ is infinite-dimensional. Then for any unconditional Gabor Schauder frame $\Gc_N(g_j,\Gamma_j)$ for $S$ with $g_j\in M^{q_j}(\R^d)$ for some $1\leq q_j\leq \frac{2p}{3p-2}$ for all $1\leq j\leq N$, there exist some subsets $\Lambda'\subseteq \Lambda$, $\Gamma_j'\subseteq \Gamma_j$ for all $1\leq j\leq N$, such that 
$\sup\limits_{x\in \Lambda'} d(x,\cup_{j=1}^N\Gamma_j')<\infty$.
\end{theorem}
\begin{proof}
    By Feichtinger's conjecture (now known as Feichtinger's theorem, for example, see \cite[Section 4.2]{Bo18}),  $\Gc(\phi,\Lambda)$ can be decomposed to into a finite union of subsequences, each of which forms a Riesz basis for its closed linear span. Since $\clspan\bigset{\Gc(\phi,\Lambda)}$ is infinite-dimensional, there exists an infinite subset $\Lambda'\subseteq \Lambda$ such that $\Gc(\phi, \Lambda')$ is a Riesz basis for its closed span. Now by Theorem \ref{characterization_uncondition_Scf}, there exist subsets $\Gamma_j'\subseteq \Gamma_j$ such that $\Gc_N(g_j,\Gamma_j')$ is a frame for $L^2(\R^d).$ The statement then follows from \cite[Corollary 4.4]{PY25c}.
\end{proof}

\begin{corollary}
\label{non_existence_unschf_subspaces}
 Let $\phi\in M^1(\R^d)$ be an arbitrary nonzero function. Assume that $S\subseteq L^2(\R^d)$ is a closed subspace that contains $\Gc(\phi,\set{0}\times\Lambda)$ for some infinite  uniformly discrete subset $\Lambda$ of $\R^{d}$. Then $S$ does not admit any unconditional Schauder frame of translates with finitely many generators.   
\end{corollary}
\begin{proof}
    It is known that $\Gc(\phi,\set{0}\times\Lambda)$ is a Bessel sequence when $\Lambda$ is uniformly discrete (for example, see \cite[Lemma 3.3]{FS07}). The statement then follows Theorem \ref{distance_criterion}.
\end{proof}

Using the fact that modulation spaces are invariant under Fourier transform, we obtain the following corollary.
\begin{corollary}
\label{uncertainty_principle_Fourier}
Let $S$ be an infinite-dimensional closed subspace of $L^2(\R^d)$ that is closed under Fourier transform. Assume that $S$ admits an unconditional Schauder frame of translates $\Gc_N(\phi_i,\Gamma_i\times\set{0})$ for some $\phi_1,\dots,\phi_N\in L^2(\R^d)$ and some subsets $\Gamma_1,\dots,\Gamma_N\subseteq\R^d$. Then 
     there exists some $1\leq i_0\leq N$ such that $\phi_{i_0}\notin M^p(\R^d)$ for any $1\leq p\leq \frac{4}{3}.$
\end{corollary}
\begin{proof}
     By assumption, $\Gc_{N}\bigparen{\widehat{\phi_i},\set{0}\times (-\Gamma_i)}$ is an unconditional Schauder frame for $S$. By Theorem \ref{characterization_uncondition_Scf} and Feichtinger Theorem, there exists some $1\leq i_0\leq N$ and an infinite subset $\Gamma_{i_0}'$ of $\Gamma_{i_0}$ such that $\Gc_{1}(\widehat{\phi}_{i_0},\set{0}\times -\Gamma_{i_0}')$ is a Riesz basis for its closed span.    
    Suppose to the contrary that $\phi_i\in M^p(\R)$ for some $1\leq p\leq \frac{4}{3}$ and all $1\leq i\leq N.$ Since modulation spaces are invariant under Fourier transform, we see that $\widehat{\phi}_{i_0}\in M^p(\R)$ for some $1\leq p\leq \frac{4}{3}$. Since $1\leq p\leq \frac{2p}{3p-2}$ if and only if $1\leq p\leq \frac{4}{3}$, we obtain a contradiction to Theorem \ref{distance_criterion}.
\end{proof}

We provide an illustrative example below how Theorem \ref{distance_criterion} improves the results in \cite{CDH99}, \cite{LT25} and \cite{PY25c}.
\begin{example}
    Let $\phi=e^{-\frac{\pi x^2}{2}}$ be the Gaussian function and let $\Lambda\subseteq \R$ be an arbitrary infinite uniformly discrete subset. Then every closed subspace containing $\Gc(\phi,\set{0}\times\Lambda)$ in $L^2(\R)$ does not admit an unconditional Schauder frame of translates. For example, let $\Lambda=a\Z\times b\Z$. It is known that $\clspan\set{\Gc(\phi,\Lambda)}=L^2(\R)$ if $0<ab\leq 1$ and that is a proper closed subspace of $L^2(\R)$ otherwise. Consequently, Corollary \ref{non_existence_unschf_subspaces} not only recovers the result in \cite{LT25} that $L^2(\R)$ does not admit any unconditional Schauder frames of translates with finitely many generators, but provides a broader class of closed subspaces that fails to admit unconditional Schauder frames of translates. 

    On the other hand, assume that $\Lambda$ is symmetric (not necessarily uniformly discrete). That is, $\Gamma=-\Gamma$. If $\clspan\set{\Gc(\phi,\Gamma)}$ admits an unconditional Schauder frame of translates, then by Corollary \ref{uncertainty_principle_Fourier}, at least one of its generators necessarily fails to belong to $M^p(\R)$ for any $1\leq p\leq \frac{4}{3}$. \qeddef  
\end{example}

\subsection{Beurling density associated with Gabor systems}
The study of Beurling density in connection with Gabor systems that form Schauder bases, Riesz bases, frames, and related structures has a long history
(\cite{RS94}, \cite{CDH99}, \cite{DH00}, \cite{Hei07}, \cite{GK24}). The central question is to characterize how densely the set of time–frequency shifts may be distributed in $\R^{2d}$. We first summarize known results in the following theorem. 
\begin{theorem}
\label{summary_density_result}
    Let $\Gc_{N}(\phi_i,\Lambda_i)$ be a Gabor system in $L^2(\R^d)$. The following statements holds:
    \begin{enumerate} \setlength\itemsep{0.5em}
        \item [\textup{(a)}](\cite{DH00}) If $\Gc_{N}(\phi_i,\Lambda_i)$ is a Schauder basis, then $D^{-}\bigparen{\cup_{i=1}^N\Lambda_i}\leq D^{+}\bigparen{\cup_{i=1}^N\Lambda_i}\leq 1$,
        \item [\textup{(b)}](\cite{CDH99}) If $\Gc_{N}(\phi_i,\Lambda_i)$ is a Riesz basis, then $D^{-}\bigparen{\cup_{i=1}^N\Lambda_i}= D^{+}\bigparen{\cup_{i=1}^N\Lambda_i}=1$,

     \item [\textup{(c)}](\cite{CDH99}) If $\Gc_{N}(\phi_i,\Lambda_i)$ is a frame, then $1\leq D^{-}\bigparen{\cup_{i=1}^N\Lambda_i}\leq D^{+}\bigparen{\cup_{i=1}^N\Lambda_i}<\infty$,

      \item [\textup{(d)}](\cite{LT25c}) If $\Gc_{N}(\phi_i,\Lambda_i)$ is a Schauder frame, then $0\leq D^{-}\bigparen{\cup_{i=1}^N\Lambda_i}\leq D^{+}\bigparen{\cup_{i=1}^N\Lambda_i}\leq\infty$,
    \end{enumerate}
  where $\cup_{i=1}^N\Lambda_i$ is the disjoint union of $\Lambda_1,\dots\Lambda_N.$
  
  In particular, all the lower and upper bounds in Theorem \ref{summary_density_result} are sharp.\qeddef
\end{theorem}
We remark that it was conjectured in \cite{DH00} that $D^{-}\bigparen{\cup_{i=1}^N\Lambda_i}= D^{+}\bigparen{\cup_{i=1}^N\Lambda_i}= 1$ if $\Gc_{N}(\phi_i,\Lambda_i)$ is a Schauder basis. This conjecture remains open. In particular, an affirmative resolution of this conjecture would also resolve the Olson–Zalik conjecture in the affirmative.

Consequently, the last missing piece in this direction is to determine the appropriate Beurling density condition for Gabor systems that form unconditional Schauder frames. Due to the absence of the upper frame inequality constraint, it is not hard to see there exist Gabor systems with infinite Beurling density that form unconditional Schauder frames (and hence Schauder frames) if one allows the associated coefficient functionals to contain infinitely many zero functions. The remaining question would be how small the associated Beurling density can be. We resolve this final aspect in the following theorem. Furthermore, we construct unconditional Schauder frames with infinite Beurling density for which all associated coefficient functionals are nonzero. The construction of such unconditional Schauder frames is inspired by \cite[Theorem 3.2]{FOSZ14}.
\begin{theorem}
\label{density_Gabor_unconditional_sf}
    Let $\Gc_N(\phi_i,\Lambda_i)$ be an unconditional Schauder frame for $L^2(\R^d).$ Then $$1\leq D^{-}\bigparen{\cup_{i=1}^N\Lambda_i}\leq D^{+}\bigparen{\cup_{i=1}^N\Lambda_i}\leq \infty,$$
    where $\cup_{i=1}^N\Lambda_i$ is the disjoint union of $\Lambda_1,\dots\Lambda_N.$
\end{theorem}
\begin{proof} By Theorem \ref{characterization_uncondition_Scf}, there exist subset $\Lambda_i'\subseteq \Lambda_i$ for all $1\leq i\leq N$ such that $\Gc_N(\phi_i, \Lambda_i')$ forms a frame for $L^2(\R^d)$. It then follows that, by Theorem \ref{summary_density_result}(c), $$D^{-}\bigparen{\cup_{i=1}^N\Lambda_i}\geq D^{-}\bigparen{\cup_{i=1}^N\Lambda_i'}\geq 1.$$ 

Next, we construct an unconditional Schauder frame with infinite Beurling density for which all associated coefficient functionals are nonzero.
    Let $\phi,\Gamma$ be such that $\Gc(\phi, \Gamma)$ forms a frame for $L^2(\R^d)$ and let $\set{\phi^*_n}\inN$ be the canonical dual frame associated with $\Gc_(\phi, \Gamma).$ Note that we have $\sup_n\norm{\phi_n^*}_{L^2(\R^d)}<\infty.$ Let $(K_n)\inN$ be an increasing sequence of positive integer such that $\sumli\frac{1}{K_n}<\infty.$ For each $n\in \N$ we pick $K_n$ distinct points $\set{(a_{n,i},b_{n,i})}_{i=1}^{K_n}$ near $(a_n,b_n)$ such that
    \begin{enumerate}
    \setlength\itemsep{0.5em}
        \item [\textup{(a)}] $\set{(a_{n,i},b_{n,i})}_{i=1}^{K_n}\cap \bigparen{\cup_{j=1}^{i-1} \set{(a_{j,i},b_{j,i})}_{i=1}^{K_j}}=\emptyset$ for all $i\geq 2,$ 
        \item [\textup{(b)}] for each $n\in \N$ we have $\norm{(a_n,b_n)-(a_{n,i},b_{n,i})}< 1/n$ for all $1\leq i\leq K_n,$
        \item [\textup{(c)}] $\norm{M_{b_n}T_{a_n}\phi-M_{b_{n,i}}T_{a_{n,i}}\phi}_{L^2(\R^d)}<(2^{2n}\sup_n\norm{\phi_n^*}_{L^2(\R^d)})^{-1}$ for all $n\in\N$ and all $1\leq i\leq K_n.$ 
    \end{enumerate}
Since $K_n$ increases to $\infty$, it follows that $D^+\bigparen{\cup_{n\in\N}\set{(a_{n,i},b_{n,i})}_{i=1}^{K_n}}=\infty.$ It now remains to show that $\cup\inN\set{M_{b_{n,i}}T_{a_{n,i}}\phi}_{i=1}^{K_n}$ is an unconditional Schauder frame for $L^2(\R^d).$ We define the operator $S\colon L^2(\R^d)\rightarrow L^2(\R^d)$ by 
\begin{equation}
\label{frame_operator_excessive_unsch}
S(x)=\sum_{n\in\N}\sum_{i=1}^{K_n}\ip{x}{K_n^{-1}\phi_n^*}M_{b_{n,i}}T_{a_{n,i}}\phi\quad \text{for all }x\in L^2(\R^d)
\end{equation}
We will show in the following that $S$ is an invertible bounded linear operator. We will also show that the series in Equation \ref{frame_operator_excessive_unsch} converges unconditionally in the norm of $L^2(\R^d).$ Once these two claims are established, it then follows that 
$$x=S(S^{-1}x)\sum_{n\in\N}\sum_{i=1}^{K_n}\ip{x}{(S^{-1})^*K_n^{-1}\phi_n^*}M_{b_{n,i}}T_{a_{n,i}}\phi\quad \text{for all }x\in L^2(\R^d),$$
with the unconditional convergence of the series in the norm of $L^2(\R^d).$ Consequently, $\bigset{M_{b_{n,i}}T_{a_{n,i}}\phi\,|\,1\leq i\leq K_n,\,n\in\N}$ is an unconditional Schauder frame for $L^2(\R^d).$

Since both sequences  $\cup_{n\in\N}\set{K^{-1/2}\phi_n^*}_{i=1}^{K_n}$ and $\cup_{n\in\N}\set{K^{-1/2}M_{b_n}T_{a_n}\phi}_{i=1}^{K_n}$ are Bessel sequences, we see that 
\begin{equation}
\label{original_frame_operator}
\sumli \sum_{i=1}^{K_n}\ip{x}{K^{-1/2}\phi_n^*}K^{-1/2}M_{b_n}T_{a_n}\phi=\sum_{n=1}^\infty \ip{x}{\phi_n^*}M_{b_n}T_{a_n}\phi=x,
\end{equation}
with the unconditional convergence of the series in the norm of $L^2(\R^d).$ By applying condition (c) above and the Cauchy–Bunyakovsky–Schwarz inequality, for each $n\in\N$ we obtain 
\begin{align}
\label{triangle_estimate_I}
\begin{split}
\norm{\ip{x}{\phi_n^*}(M_{b_{n,i}}T_{a_{n,i}}-M_{b_n}T_{a_n})\phi}_{L^2(\R^d)}&\Le \sup_n\norm{\phi_n^*}_{L^2(\R^d)}\norm{x}_{L^2(\R^d)}(2^{2n}\sup_n\norm{\phi_n^*}_{L^2(\R^d)})^{-1}\\
&\Le 2^{-2n}\norm{x}_{L^2(\R^d)}.
\end{split}
\end{align}
By the identity $M_{b_{n,i}}T_{a_{n,i}}\phi=(M_{b_{n,i}}T_{a_{n,i}}\phi-M_{b_n}T_{a_n}\phi)+M_{b_n}T_{a_n}\phi$, Equation (\ref{triangle_estimate_I}) and Triangle Inequality, it follows that, for any positive integers $N<M$ and any $x\in L^2(\R^d)$ with $\norm{x}_{L^2(\R^d)}=1$, we have 
\begin{align*}
\begin{split}
   \Bignorm{\sum_{n=N}^M\sum_{i=1}^{K_n}\ip{x}{K_n^{-1}\phi_n^*}M_{b_{n,i}}T_{a_{n,i}}\phi}_{L^2(\R^d)}\Le  \sum_{n=N}^M2^{-2n}+  \Bignorm{\sum_{n=N}^M\sum_{i=1}^{K_n}\ip{x}{K_n^{-1}\phi_n^*}M_{b_{n}}T_{a_{n}}\phi}_{L^2(\R^d)}.   
    \end{split}
\end{align*}
Consequently, $S(x)$ is a well-defined linear operator. Using the Triangle Inequality again and arguing similarly, we see that $$\norm{(S-I)x}_{L^2(\R^d)}=\Bignorm{\sum_{n\in\N}\sum_{i=1}^{K_n}\ip{x}{K_n^{-1}\phi_n^*}M_{b_{n,i}}T_{a_{n,i}}\phi-\sum_{n=1}^\infty \ip{x}{\phi_n^*}M_{b_n}T_{a_n}\phi}_{L^2(\R^d)}<1,$$
for all $x\in L^2(\R^d)$ with $\norm{x}_{L^2(\R^d)}=1$. Hence, $S$ is an invertible bounded linear operator. Moreover, the series in Equation (\ref{frame_operator_excessive_unsch}) converges unconditionally due to the unconditional convergence of the series in Equation (\ref{original_frame_operator}). \qedhere

\end{proof}

\subsection{Uncertainty Principle} It is known that there exists a certain incompatibility between Gabor systems with ``nice windows" at critical density and ``nice" approximation properties. In light of Theorem \ref{summary_density_result}, a Gabor system $\Gc_{N}(\phi_i,\Lambda_i)$ is said to have critical density if $$D^{-}\bigparen{\cup_{i=1}^N\Lambda_i}=D^{+}\bigparen{\cup_{i=1}^N\Lambda_i}=1.$$
The classical Balian-Low Theorem states that the Gabor system $\Gc(\phi,\Z\times \Z)$ cannot simultaneously be a Riesz basis and have a window function that belongs to the Feichtinger algebra. Therefore, window functions for Gabor Riesz bases $\Gc(\phi,\Z\times \Z)$ cannot exhibit ``good" joint time-frequency localization. For generic sets of time-frequency shifts, Ascensi, Feichtinger and Kaiblinger proved in \cite{AFK14} that no Gabor Riesz basis at critical density can be generated using windows functions from the Feichtinger algebra. In fact, they showed that the same conclusion holds for Gabor frames with critical density. We futher confirm in the following theorem that no Gabor systems at critical density can simultaneously be an unconditional Schauder frame and have window functions in the Feichtinger algebra. We refer to \cite{BHW95},\cite{CP06},\cite{GM13},\cite{AFK14}, and the references therein for further background and recent developments on this topic.  
\begin{theorem}
\label{M1_window_strict_lower_Beurling}
Let $\Gc_N(\phi_i,\Lambda_i)$ be an unconditional Schauder frame for $L^2(\R^d)$. Assume that $\phi_i\in M^1(\R^d)$ for all $1\leq i\leq N$. Then $D^{-}\bigparen{\cup_{i=1}^N\Lambda_i}>1,$  where $\cup_{i=1}^N\Lambda_i$ is the disjoint union of $\Lambda_1,\dots\Lambda_N.$
\end{theorem}
\begin{proof}
    Suppose to the contrary that $D^{-}\bigparen{\cup_{i=1}^N\Lambda_i}\leq 1.$ By Theorem \ref{characterization_uncondition_Scf}, there exist some subsets $\Lambda_i'\subseteq \Lambda_i$ for all $1\leq i\leq N$ such that $\Gc_N(\phi_i,\Lambda_i')$ forms a frame for $H$. However, the assumption implies that $$D^{-}\bigparen{\cup_{i=1}^N\Lambda_i'}\Le D^{-}\bigparen{\cup_{i=1}^N\Lambda_i}\leq 1,$$
    which is a contradiction to \cite[Theorem 1.5]{AFK14}.
\end{proof}
This result is optimal in the sense that one can not expect Theorem \ref{M1_window_strict_lower_Beurling} to hold for window functions that do not all belong to $M^1(\R^d).$ For example, $\Gc(\chi_{[0,1)^d}, \Z^{2d})$ is an orthonormal basis for $L^2(\R^d)$ with $D^{-}(\Lambda)=D^{+}(\Lambda)=1$, while $\chi_{[0,1)^d}\in M^p(\R^d)$ for any $p>1.$

The Bargmann-Fock space $\Fc^2(\C^d)$ is the Hilbert space consisting of all entire functions $f$ on $\C^d$ for which the norm $$\norm{f}_{\Fc^2(\C^d)}^2\Eq \int_{\C^d} |f(z)|e^{-\pi |z|^2}\,dz$$
is finite.
It was shown by Gr\"{o}chenig and Malinnikova in \cite{GM13} that sequences of the form $\set{e^{\pi\bar{\lambda}z}}_{\lambda\in \Lambda}$ is never a Reisz basis for $\Fc^2(\C^d)$ regardless of the choice of $\Lambda\subseteq \C^d$. Ascenci et al.\ then improved this result by showing that $D^{-}(\Lambda)>1$ if $\set{e^{\pi\bar{\lambda}z}}_{\lambda\in \Lambda}$ (see \cite[Remark 1.9]{AFK14}) is a frame for $\Fc^2(\C^d)$.
Let $\phi(x)=2^{\frac d4}e^{-\pi \norm{x}^2}$ be the Gaussian function in $d$ variables and $\Lambda\subseteq \C^d$ be a countable subset. It is known that frames of the form $\set{e^{\pi\bar{\lambda}z}}_{\lambda\in \Lambda}$ (or equivalently, \emph{sampling sequences}) for $\Fc^2(\C^d)$ are unitarily equivalent to Gabor frames of the form $\set{M_{-b}T_a\phi}_{(a+ib)\in \Lambda}$ for $L^2(\R^d)$ via Bargmann transform (for example, see \cite[Section 3.4]{Gro01}).  Using the Theorem \ref{characterization_uncondition_Scf}, Theorem \ref{M1_window_strict_lower_Beurling} and the fact that $\phi\in M^1(\R^d)$, we obtain the following corollary. Here the Beurling density of a countable subset of $\C^d$ is defined by the obvious modification of \ref{Beruling_density}.


\begin{corollary}
    Assume that $\Lambda$ is a countable subset of $\C^{d}$ such that $\set{e^{\pi\bar{\lambda}z}}_{\lambda\in \Lambda}$ forms an unconditional Schauder frame for the Bargmann-Fock space $\Fc^2(\C^d)$. Then $D^{-}(\Lambda)>1.$\qeddef
\end{corollary}

Next, we present applications to exponential systems in $L^2(S)$, where $S$ is an arbitrary subset of $\R^d$ with positive finite (Lebesgue) measure. 
\subsection{Exponential Systems} Let $\Lambda$ be a countable subset of $\R^d$. The exponential system associated with $\Lambda$ is the set $$\Ec(\Lambda)=\set{e^{2\pi i\lambda\cdot x}}_{\lambda\in \Lambda}.$$ Landau's necessary density conditions states that $D^{-}(\Lambda)\geq |S|$ whenever $\Ec(\Lambda)$ forms a frame for $L^2(S)$, and that $D^{+}(\Lambda)\leq |S|$ if $\Ec(\Lambda)$ forms a Riesz basis for its closed span. Consequently, $D^{-}(\Lambda)=D^{+}(\Lambda)=|S|$ if $\Ec(\Lambda)$ is a Riesz basis for $L^2(S).$ As in the case of Gabor systems, it remains unknown how small $D^-(\Lambda)$ can be when $\Ec(\Lambda)$ forms an unconditional Schauder frame for $L^2(S).$ By Theorem \ref{characterization_uncondition_Scf}, we obtain the following necessary density condition for such exponential systems. We remark that, by an argument similar to that used in Theorem \ref{density_Gabor_unconditional_sf}, one can also construct $\Lambda\subseteq\R^d$ with $D^+(\Lambda)=\infty$ such that $\Ec(\Lambda)$ forms an unconditional Schauder frame that possess coefficient functionals that are all nonzero.

\begin{theorem} Let $S\subseteq \R^d$ be a subset with positive finite measure. Assume that $\Ec(\Lambda)$ is an unconditional Schauder frame for $L^2(\R^d)$. Then $|S|\leq D^{-}\bigparen{\Lambda}\leq D^{+}\bigparen{\Lambda}\leq \infty.$ 
\end{theorem}
It was an open problem whether every $L^2(S)$ admits a Riesz basis of exponential system. The first example of a set $S$(in fact, a compact set) for which $L^2(S)$ does not admit any Riesz basis of exponentials was constructed by Kozma, Nitzan and Olevskii in \cite{KNO23}. Unlike Riesz bases, every $L^2(S)$ admits an exponential frame. It is then natural to ask whether every $L^2(S)$ admits an exponential frame $\Ec(\Lambda)$ with $D^{-}(\Lambda)=|S|.$ Recently, Enstad and van Velthoven proved that the compact set $S$ constructed by Kozma et al.\ in \cite{KNO23} does not admit any exponential frame $\Ec(\Lambda)$ with $D^{-}(\Lambda)=|S|.$ We further show that the same conclusion holds for unconditional Schauder frames of exponentials.

Combining Theorem \ref{characterization_uncondition_Scf} with \cite[Theorem 1.1(i)]{EV25}, we obtain the following theorem. Recall that we say that a sequence of subsets of $\R^d$, $\set{\Lambda_n}\inN$, \emph{converges weakly} to $\Lambda\subseteq \R^d$ if for every $\epsilon>0$ and every $r>0$ there exists $N\in\N$ such that $$\Lambda_n\cap B_{r}(0)\subseteq \Lambda+B_{\epsilon}(0)\quad \text{and} \quad \Lambda \cap B_{r}(0)\subseteq \Lambda_n+B_{\epsilon}(0),$$
for all $n\geq N.$ A subset $\Lambda'$ is said to be a weak limit of translates of $\Lambda\subseteq \R^d$ if there exists some sequence of scalars $(a_n)\inN$ such that $\Lambda-a_n$ converges weakly to $\Lambda'.$ We denote by $W(\Lambda)$ the set of all weak limits of translates of $\Lambda\subseteq \R^d$. For any subsets $M,N$ of $\R^d$, $M+N$ and $M-N$ denote the usual Minkowski sum and difference, respectively. 
\begin{theorem} Let $S\subseteq \R^d$ be a subset with positive finite measure and let $\Lambda \subseteq \R^d$ be a uniformly discrete subset. Assume that $\Ec(\Lambda)$ is an unconditional Schauder frame for $L^2(S)$ with $D^-(\Lambda)=|S|$. Then there exists some subset $\Lambda'\subseteq\Lambda$ and some subset $\Gamma\subseteq W(\Lambda')$ such that $\Ec(\Gamma)$ is a Riesz basis for $L^2(S).$
\end{theorem}
\begin{proof}
    Assume that $D^-(\Lambda)=|S|$. Then by Theorem \ref{characterization_uncondition_Scf} and Landau's necessary density condition, there exists some subset $\Lambda'\subseteq \Lambda$ with $D^-(\Lambda)=|S|$ such that $\Ec(\Lambda')$ forms a frame for $L^2(S).$ The statement then follows by \cite[Theorem 1.1(i)]{EV25}.
\end{proof}
Arguing similarly to the one used in \cite[Theorem 1.2(i)]{EV25}, we obtain the following result.
\begin{theorem} There exists some compact subset $S\subseteq \R$ of positive measure such that for any countable subset $\Lambda$
with $D^{-}(\Lambda)=|S|$ the exponential system $\Ec(\Lambda)$ is never an unconditional Schauder frame for $L^2(S)$.\qeddef
\end{theorem}

Finally, we conclude this paper with an application to iterative systems.
\subsection{Frame-normalizability of iterative systems generated by normal operators} Let $A\colon H\rightarrow H$ be a (bounded) normal operator on $H$ and let $S\subseteq H$ be a countable subset. It was conjectured by Cabrelli et al.\ in \cite{ACMCP17} that the iterative system $\set{A^nx}_{x\in S,\,n\geq 0}$ can never be normalized to form a frame for $H.$ That is, $\bigset{\frac{A^nx}{\norm{A^nx}}}_{x\in S,\,n\geq 0}$ is never a frame for $H.$ For research papers related to this conjecture, we refer to \cite{ACMT17}, \cite{ACMCP17}, \cite{PY24} and \cite{PY25d}. 
This conjecture remains open at the time of writing. Even surprisingly, it is still unknown whether there exist some normal operator $A$, an element $x\in H$ and some infinite subset $\Gamma\subseteq \N\cup\set{0}$ such that $\bigset{\frac{A^nx}{\norm{A^nx}}}_{n\in \Gamma}$ forms a frame for $H.$ As a corollary of Theorem \ref{characterization_uncondition_Scf}, we establish the existence of such $A,\Gamma$ and element $x\in H$ in the following theorem. In particular, this result also shows that the conjecture by Cabrelli et al.\ is false if one relax the set iterations from $\N\cup\set{0}$ to any arbitrary infinite subset of $\N\cup\set{0}$. 
\begin{theorem}
    There exist some normal operator $A\colon H\rightarrow H$, some $x\in H$ and some infinite subset $\Gamma\subseteq \N\cup\set{0}$ such that $\bigset{\frac{A^nx}{\norm{A^nx}}}_{n\in \Gamma}$ forms a frame for $H.$ 
\end{theorem}
\begin{proof}
    Since every frame for $H$ is an unconditional Schauder frame, by Theorem \ref{characterization_uncondition_Scf}, it remains to find some normal operator and some $x\in H$ such that $\set{A^nx}_{n\geq0}$ is a frame for $H.$ Such characterizations can be found in \cite[Theorem 5.7]{ACMCP17} (See also \cite{CMPP20}, \cite{AP17} and \cite{ACNP25}).
\end{proof}

\section*{acknowledgement}
We thank Marcin Bownik for fruitful discussions on this project.

\end{document}